\documentclass[]{article}

\newcommand{\nwc}{\newcommand}
\nwc{\draftdate}{\today}
\usepackage{datetime}
\settimeformat{ampmtime}
\usdate

\usepackage[dvips]{pstricks}
\usepackage{wrapfig}
\usepackage{epsf}
\usepackage{amsmath,amssymb}
\usepackage{amsthm}
\usepackage{CJK}
\usepackage{color}
\usepackage{colordvi}
\usepackage{graphicx}

\renewcommand{\theenumi}{{\rm(\roman{enumi})}}

\newtheorem{theorem}{Theorem}[section]
\newtheorem{lemma}[theorem]{Lemma}

\newtheorem{remark}[theorem]{Remark}

\newtheorem{corollary}[theorem]{Corollary}

\newcommand{\disp}{\displaystyle}
\nwc{\qref}[1]{(\ref{#1})}

\newcommand{\beq}{\begin{equation}}
\newcommand{\eeq}{\end{equation}}

\def\XXint#1#2#3{{\setbox0=\hbox{$#1{#2#3}{\int}$}
\vcenter{\hbox{$#2#3$}}\kern-.51\wd0}}

\newcommand{\vep}{\varepsilon}

\newcommand{\cA}{{\cal A}}

\newcommand{\cF}{{\cal F}}

\newcommand{\cG}{{\cal G}}

\newcommand{\cS}{{\cal S}}

\def\kgam{\kappa}
\newcommand{\N}{\mathbb N}
\newcommand{\R}{\mathbb R}

\newcommand{\tF}{\tilde{F}}

\newcommand{\tu}{\tilde{u}}

\newcommand{\tw}{\tilde{w}}

\nwc{\bb}{{\bf b}}
\nwc{\bC}{{\bf C}}
\nwc{\bB}{{\bf B}}
\nwc{\bF}{{\bf F}}
\nwc{\bJ}{{\bf J}}
\nwc{\bk}{{\bf k}}
\nwc{\bm}{{\bf m}}
\nwc{\bn}{{\bf n}}
\nwc{\bN}{{\bf N}}
\nwc{\bfr}{{\bf r}}
\nwc{\bt}{{\bf t}}
\nwc{\bT}{{\bf T}}
\nwc{\bv}{{\bf v}}
\nwc{\bV}{{\bf V}}
\nwc{\bw}{{\bf w}}
\nwc{\bz}{{\bf z}}
\nwc{\vkappa}{{\vec{\kappa}}}
\nwc{\vr}{\vec{r}}
\nwc{\eps}{\varepsilon}

\nwc{\ip}[1]{\langle #1 \rangle}
\nwc{\ipbig}[1]{\left\langle #1 \right\rangle}
\newcommand{\be}{\begin{align}}
\newcommand{\ee}{\end{align}}
\newcommand{\nn}{\nonumber}
\newcommand{\ben}{\begin{align*}}
\newcommand{\een}{\end{align*}}

\title{Codimension One Minimizers of Highly Amphiphilic Mixtures}

\author{Shibin Dai\footnote{S.D. acknowledges support of NSF grants DMS-1802863 and DMS-1815746.} \\
Department of Mathematics, 
          The University of Alabama, \\
          Box 870350, Tuscaloosa, AL 35487-0350, USA \\
          \and Keith Promislow\footnote{K.P. acknowledges support
of NSF grants DMS-1409940 and DMS-1813203.} \\
          Department of Mathematics, 
          Michigan State University, \\
          East Lansing, MI 48824, USA}

\begin{document}
\maketitle


\maketitle
\begin{abstract}

We present a modified form of the Functionalized Cahn Hilliard (FCH) functional which
models highly amphiphilic systems in solvent.  A molecule is highly amphiphilic if the energy of 
a molecule isolated within the bulk solvent molecule is prohibitively high. 
For such systems once the amphiphilic molecules assemble into a structure it is very rare for a molecule 
to exchange back into the bulk.  The highly amphiphilic FCH functional has a well with limited
smoothness and admits compactly supported critical points. In the limit of molecular length $\vep\to0$
we consider sequences with bounded energy whose support resides within an $\vep$-neighborhood
of a fixed codimension one interface. We show that the FCH energy is uniformly bounded below, independent 
of $\vep>0$,
and identify assumptions on tangential variation of sequences that guarantee the existence of subsequences 
that converge 
to a weak solution of a rescaled bilayer profile equation, and show that sequences with limited tangential variation
enjoy a $\liminf$ inequality. For fixed codimension one interfaces we construct bounded energy sequences
which converge to the bilayer profile and others with larger tangential variation which do not converge to 
the bilayer profile but whose limiting energy can violate the $\liminf$ inequality, depending upon 
the energy parameters.
\end{abstract}



\section{Introduction}

Amphiphilic molecules play an essential role in the self assembly of nano-scale structures in solvent,
in biological context they play an essential role in the formation of cell membranes and other organelles, 
and they are
increasingly important in applications of synthetic chemistry. There are distinct approaches to model the 
free energy of amphiphilic mixtures 
that emphasize different scalings and assumptions on morphology. The classical sharp interface 
approximations include the 
Canham-Helfrich energy  \cite{Canham-70, Helfrich-73} which characterizes the free energy of a codimension 
one interface embedded in
$\R^3$ in terms of its two curvatures. This is an appealingly simple formulation but does not readily 
handle singularities associated with topological change.  Conventional phase field models based on the 
Cahn-Hilliard (CH) energy
describe single layer interfaces, see \cite{dlw:retrieving, Chun-06, Loreti-00, Low2-09, Roger-2006, 
ryham:open-vesicles, Low-09} and references therein.  Single-layer interfaces separate dissimilar
phases which can not be merged and describe the volume of the boundary between the phases as if it were a 
high-energy void. For oil and water blends this is a very reasonable approximation, as a hydrophobic molecule 
generates a cavity when placed within water, \cite{Wiebe-12} and see Figure\,\ref{fig:simulation}(left).  Amphiphilic 
materials can create interfaces between similar fluids  \cite{Frederix-2018}, 
see Figure\,\ref{fig:simulation}(right), or 
reside at interfaces between dissimilar fluids, such as oil and water. In this latter case they are also
called surfactants and lower the overall  mixture energy by packing the hydrophobic cavity.  
Modeling amphiphilic interfaces with single layer energies  presents certain limitations, the first is that 
single layer models allow the volume of interface to change without material transport. 
Even when total interfacial volume is preserved by constraint, interfacial volume removed at one 
point may reappear at another distant point 
without requiring transport of surfactant. In many applications transport of surfactant is the most significant rate 
limiting step, this is particularly
true of highly amphiphilic molecules.  A second limitation of single layer models is that the single layer 
interface can not be punctured. The opening of holes in vesicles
requires the introduction of additional order parameters for {\sl each} vesicle, in particular the total number 
of vesicles 
must be predetermined, \cite{ryham:open-vesicles}.
 
 \begin{figure}[h]\label{fig:simulation}
	\centering
	\begin{tabular}{ccc}
	\includegraphics[height=4.5cm]{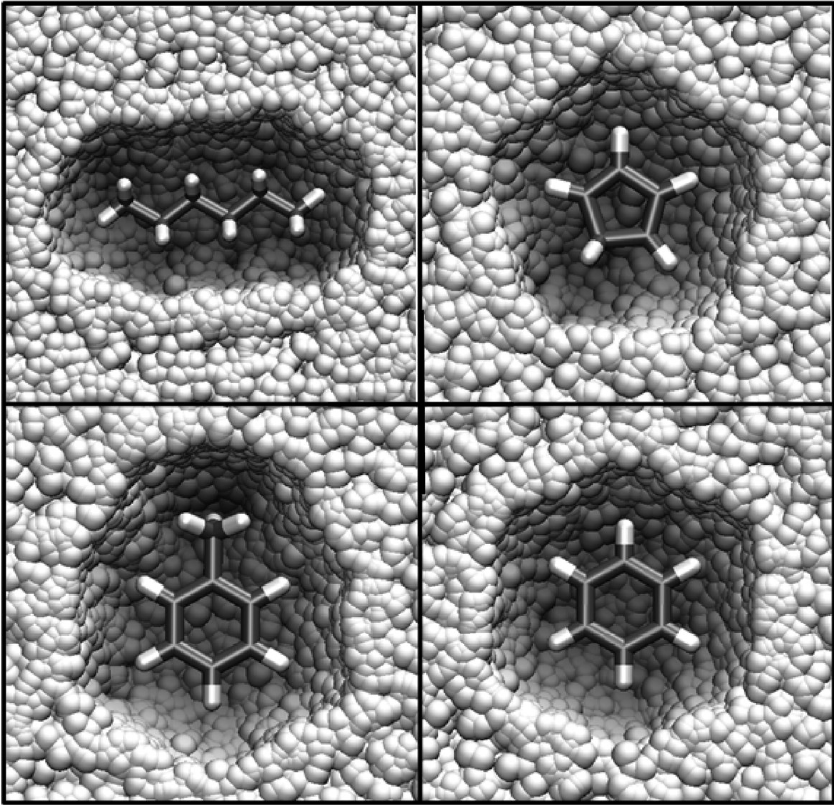}&\hspace*{0.2in}&
	\includegraphics[height=4.5cm]{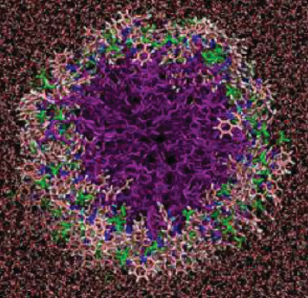}
	\end{tabular}
	\caption{(left) Results of molecular dynamics simulation showing shape of void induced in solvent phase
	in response to hydrophobic molecules of different shapes. 
	Reprinted with permission from \cite{Wiebe-12}
	Copyright (2012) American Chemical Society. 
	(right) All atom simulation of packing of 
	amphiphilic molecule (grey) at interface between external solvent molecules (reddish) and internal 
	solvent (not imaged to show internal structure). Reproduced from \cite{Frederix-2018} with 
	permission from the Royal Society of Chemistry.
	}
\end{figure}

The functionalized Cahn-Hilliard  (FCH) free energy is a phenomenological 
model describing the free energy of amphiphilic mixtures 
\cite{ckp:competition, dp:bilayer, dp:pore, promislow-2014, promislow-2011, promislow-2012,  nk-kp-2018},
that supports codimension one bilayer interfaces that separate two identical phases by a thin region of a  second 
phase -- the amphiphilic material. 
Bilayer interfaces can be punctured and can form free edges and open structures. Gradient flows of the FCH 
free energy 
transport amphiphilic molecules either along the interface or through the bulk solvent phase.  Over a bounded 
domain $\Omega\subset \R^n$, the conventional Cahn-Hilliard 
functional takes the form 
\begin{eqnarray}\label{def-E}
   E_\vep(u) :=\int_\Omega \left(\frac{\vep}{2}|\nabla u|^2 + \frac1\vep  W(u)\right)\;dx, 
\end{eqnarray}
where the double well potential $W:\R\mapsto\R$ has two equal-depth minimizers at $u=0$ and $u=u_+>0$ 
corresponding to the two phases, and $\vep>0$ is a small parameter characterizing the width of the void region 
between the phases. The weak functionalization form of the FCH functional can be scaled as
\begin{align}\label{def-F}
   F_\vep(u) := \int_\Omega\left\{\frac1{2\vep}\left( -\vep\Delta u + \frac1\vep  W'(u)\right)^2 
      - \left(\frac{\eta_1\vep}{2}|\nabla u|^2 + \frac{\eta_2}{\vep} W(u) \right)  \right\}\;dx.
\end{align}
The first term on the left-hand side of (\ref{def-F}) is the square of the variational derivative of a Cahn-Hilliard 
type energy and its volume integral captures distance of the configuration $u$ from a critical 
point of the underlying 
Cahn-Hilliard type functional -- these critical points represent the optimal packings of the amphiphilic molecules. 
The double well has unequal depth minima, $W(0)=0>W(u_+)$. The parameters $\eta_1>0$ and $\eta_2\in\R$ 
model the strength of hydrophilic interactions and the aspect ratio of the amphiphilic molecule respectively.  
For fixed $\vep>0$ and mass amphiphilic mass fraction $M:=\int_\Omega u\,dx$ 
the functional $F_\vep$ is bounded from below with a negative lower bound that diverges to negative infinity as 
$\vep\to 0$ \cite{PZ-13}.  

A key experimental quantity is the exchange rate which governs the probability of an 
amphiphilic molecule being ejected from an interface and returning to the bulk solvent. 
This rate is governed by the 
difference in free energies of a molecule when it is in the interface and when it is in the bulk. The insertion 
rate, that is, 
the propensity of an amphiphilic molecule in the bulk to be absorbed onto an interface, typically varies inversely 
to the 
exchange rate. These two processes mediate the exchange of surfactant materials between disjoint 
interfaces.  The experimental 
literature shows that the exchange rate decays exponentially with the length of the hydrophobic 
moiety \cite{bates-2003}, 
diminishing effectively to zero for sufficiently diblocks with sufficiently long hydrophobic regions. 
These ``highly amphiphilic'' surfactants 
have a prohibitively large solvation energy and produce structures that have almost no measurable 
exchange of amphiphilic molecules 
with the bulk. The structures they form act like isolated islands in a sea of solvent. Important classes of 
amphiphic are considered to be highly amphiphilic, 
these include most lipids which are characterized by a small hydrophilic head and longer, 
double-branched hydrophobic tail. 

Analysis of mass-preserving gradient flows of the FCH  functional has focused on the stability and 
asymptotic dynamics  of codimension one bilayers and codimension two filamentous pore structures. 
The analysis shows that 
the temporal  rate of surfactant exchange is inversely proportional to the second derivative of the double well 
at the pure solvent 
phase, $W''(0)$. This is relatively transparent from the form of the FCH free energy \qref{def-F}, 
for which a spatially constant 
distribution $u\equiv c$ has an energy
\beq
 F_\vep(c)=  \frac12\vep^{-3}c^2 |\Omega|\cdot |W''(0)|^2 + O(\vep^{-3}c^3).
 \eeq 
Moreover interfaces decay exponentially in space to the bulk constant value at a rate inversely proportional to 
$\vep\sqrt{W''(0)}.$ The exchange rate of an interfacial structure depends primarily upon 
its codimension. The differing exchange rates between structures of different codimension generically 
leads to growth of structures whose codimension has the lowest associated 
exchange rate at the expense of those structures with higher exchange rates 
\cite{dp:bilayer, dp:pore, promislow-2014, promislow-2011, promislow-2012, nk-kp-2018}.  To model highly 
amphiphilic molecules we drive the exchange rate to zero by considering a double well which is smooth for 
$u>0$ but 
only $C^{1+\alpha}$ in a neighborhood of $u=0$. This generates
minimizers of the FCH that are potentially compactly supported, with no mechanism to exchange 
surfactant molecules 
between disjoint structures. 

In experimental situations the amphiphilic materials typically occupy a small fraction of the total volume of 
the domain.  
Indeed, to produce an $O(1)$ surface area of codimension one bilayer requires an $O(\eps)$ volume 
of amphiphile. 
Given the compactly supported nature of the critical points
of the FCH energy it is natural to consider distributions  $u\in H^2(\Omega)$ of amphiphile whose support lies in a 
small subset $\Omega_\vep$ within the total domain $\Omega$. 
Considering the boundary $\partial \Omega$ of the domain to be unfavorable for amphiphilic molecules, 
for example a clean 
glass beaker,  it is
natural to impose no-contact, no-flux boundary conditions 
\begin{align}\label{boundary-cond}
	u = 0, \quad \frac{\partial u}{\partial\nu}= 0,
\end{align}
where $\nu$ is the outer normal of $\partial\Omega$.  These boundary conditions indicate that amphiphilic 
structures lie in $\Omega$ and away from the boundary 
$\partial \Omega$. We further require the local regularity assumption
\beq
	W(u) \sim u^r  \mbox{ for some } 3/2<r<2 \mbox{ as }u\to 0^+, \label{A-W1}\\
\eeq
and the growth rate assumptions
\begin{align}
		C_1|u|^p +C_2 \leq &W(u) \leq  C_1|u|^p + C_3,  	\label{G-W1}\\
		&|W'(u)|\leq C_1 p|u|^{p-1} + C_3', \label{G-W2} \\
		 C_1 p |u|^p  +C_4\leq &W'(u) u. 		 \label{G-W3}
 \end{align}
for some constants $C_1>0$ and $C_2,C_3, C_3',C_4\in\R$ and all $u\in\R$. 
Here  
\begin{align} \label{p-bound}
	2\leq p<\infty \;\;\mbox{ if  } n=2\qquad \mbox{ and  } \qquad 2\leq p<\frac{2n-2}{n-2} \;\;\mbox{ if } n\geq 3. 
\end{align}
Specifically for $n=3$, we require $2\leq p<4$. These conditions imply that
\[W(0)=W'(0)=0,  \quad W''(0^+)=+\infty. \] The requirement $r>3/2$ guarantees that the solvent-free profile, 
$u\equiv0,$ is a critical point of $F_\vep$, and the requirement $r<2$ 
signals the highly amphiphilic nature of the surfactant and guarantees the existence of compactly 
supported bilayer 
profiles corresponding to critical points of $F_\vep$.  The growth rate requirements on $p$ are technical 
considerations 
to establish the existence of minimizers, the value of $p$ has little impact on the model and no physical 
significance. A 
generic example of a double well satisfying these requirements is
 \beq
 \label{W-def}
 W(u) =\chi(u) |u|^{r}\left((u-u_+)^2+\tau \left(u- \frac{1+r}{r}u_+\right)\right) + C_5(1-\chi(u))|u|^p,
 \eeq 
 where $\chi:\R\mapsto \R$ is a $C^\infty$ cut-off function which is $1$ on $[-1, 2u_+]$ and zero 
 outside a compact set.
 The parameter $C_5>0$ is chosen large enough to guarantee that $W'$ has no zeros outside of $[0,u_+].$ 
 The parameter $\tau$ controls the depth of the right well: $W(u_+) = - \frac{\tau}{r}|u_+|^{1+r}<0$ for 
 $\tau>0.$

\begin{remark} \label{remark-1}
Finding local minimizers of $F_\vep$ under the restriction that $u\geq 0$ is both analytically and numerically 
challenging as the 
variational derivative of $F_\vep$ involves $W''$ which is not well defined at 0. One approach is to use 
variational inequalities,
as outlined in \cite{kinderlehrer:var-ineq}. While the model here differs from degenerate mobility ones 
considered for the Cahn Hilliard equation,   
\cite{cen:CH, dd:onesidedCH, dd:degenerateCH,  dd:numericCH, dd:weak,  Elliott-96, lms:degenerateCH},
the mechanism that prevents interaction through the bulk is  fundamentally distinct. For degenerate 
mobility molecules in the bulk are frozen in place and cannot move. In the model presented here 
molecules in the bulk phase would readily move and be  rapidly and permanently absorbed in finite time;
once depleted there is no mechanism to replenish the bulk density and bulk diffusion ceases.  
\end{remark}

\subsection{Summary of main results}
This discussion motivates the introduction of ``geometrically localized critical points,'' $u_\vep$ of $F_\vep$ over
the set of functions in $H^2(\Omega)$ whose support is contained within an open subset $\Omega_\vep$ 
compactly contained in $\Omega$.  These critical points are constructed in section 2 for fixed values of $\vep>0$. 
In section 3 for fixed $\ell>0$ we take $\Omega_\vep$ to be the $\vep\ell$ neighborhood of a fixed 
codimension one 
interface $\Gamma$ compactly embedded in $\Omega$, and analyze behaviors of sequences of functions 
$\{u_{\vep_k}\}_{k=1}^\infty$ with masses $\{m_{\vep_k}\}_{k=1}^\infty$ whose
energies $\{\cF_{\vep_k}(u_{\vep_k})\}_{k=1}^\infty$ are uniformly bounded as $\vep_k\to 0$ as $k\to\infty.$ 
We call 
these {\sl codimension-one $\ell$-bounded sequences} with respect to $\Gamma$, and establish upper 
bounds on their 
through plane and tangential  derivatives. In Theorem 3.1, under the assumption of slightly stronger bounds 
on the tangential derivatives, \qref{T1-B2},
we establish the existence of a subsequences which converge to a weak solution of the codimension one 
bilayer equation
\beq
  U_{zz} = W'(U),
\eeq
on a rescaled domain $\Omega_1$ where $z$ is $\vep$-scaled distance to $\Gamma$. In section 4, in 
Theorem 4.1 we establish
a $\liminf$ inequality that holds for all codimension-one $\ell$ bounded sequences that satisfy a yet 
slightly stronger bound, \qref{T2-B3}, on the tangential derivatives.
The $\liminf$ inequality provides a lower bound in terms of a $\sup$ of codimension-one energy, $\cG_1$, 
defined in \qref{def-CD1-Energy}, evaluated at the weak solutions
of the bilayer equation, see Corollary 4.2. Moreover for general codimension-one interfaces $\Gamma$ 
we construct codimension-one $\ell-$bounded sequences composed of
solutions of the bilayer equation which converge in a strong sense to the codimension-one energy $\cG_1$, 
achieving the lower bound. 
However, for the same interfaces, we construct a second family of codimension-one $\ell$ bounded 
sequences composed of superpositions of codimension $n$ 
critical points, called micelles when $n=3$,  with disjoint, compact support. These sequences do not satisfy 
the enhanced bounds on the tangential derivatives, and
do not have a subsequence converging to the bilayer equation, but their energies converge to a limit that may 
be lower than the associated codimension one limit, depending upon the 
choice of the functionalization parameters $\eta_1$  and $\eta_2$ and the size of the curvatures of $\Gamma$. 

\section{Geometrically localized minimizers of  $F_\vep$} 

Compactly supported solutions and positive solutions for reaction-diffusion equations have been of interest for 
theoretical and application reasons, see \cite{Diaz-15} and references therein. In \cite{Diaz-15} it has been shown 
that if the domain $\Omega$ is sufficiently large, then a class of low regularity reaction-diffusion equations 
admits compactly supported positive solutions, which can be interpreted as 
compactly supported positive critical points for free energies of the form (\ref{def-E}), with $\vep =1$ and 
a nonsmooth potential 
\[ W_{\alpha,\lambda}(u) = \frac1{\alpha+1} |u|^{\alpha+1} -\frac{\lambda}2u^2,\] 
for any $0<\alpha<1$, and $\lambda>\lambda_1$, where $\lambda_1>0$ is the first eigenvalue for the 
Laplace operator with 
Dirichlet boundary conditions. In our context a 
large domain is equivalent to a sufficiently small value of $\vep$. Although the
potential $W_{\alpha,\lambda}$ approaches negative infinity as $|u|\to \infty$, it 
shares the same 
strong absorption property as our non-smooth potential $W$, namely, 
$W'_{\alpha,\lambda}(0)=0, \;W''_{\alpha,\lambda}(0^+)= +\infty$.

In this section, we will show that for any subset $\Omega_\vep\subset\Omega$ such that 
$\partial\Omega_\vep$ is $C^1$,  $F_\vep$ has a minimizer $u_\vep$ over the class of 
nonnegative functions in $H^2(\Omega)$ with support inside of $\Omega_\eps$ subject to a prescribed 
total mass of the lipid phase, 
\beq
\label{mvep-def}
\int_{\Omega_\vep} u\,dx = m_\vep.
\eeq
Since $F_\vep(0)=0$, we need only consider the integral of the FCH energy density over the subset 
$\Omega_\vep$, 
denoted by $F_\vep|_{\Omega_\vep}$, 
over all $u$ in the admissible set
\begin{align}\label{def-A_vep}
	\cA_\vep:= &\left\{ u\in H_0^2(\Omega_\vep):\quad u\geq 0 \mbox{ in }\Omega_\vep, \;\;
	\int_{\Omega_\vep}u\,dx 
		= m_\vep \right\}.
\end{align}
To construct $u_\vep$,  we first derive a lower bound of $F_\vep|_{\Omega_\vep}$ over all 
$u\in H_0^2(\Omega_\vep)$. 
The proof is a modification
of that \cite{PZ-13}, which incorporates the positivity assumption that allows an explicit formulation of the 
lower bound on $\vep$.  
We highlight the key steps of the calculation.

\begin{lemma} \label{Lemma-lowerbound}
	Suppose $W$ is a double well potential satisfying (\ref{A-W1}) and  the growth assumptions 
	(\ref{G-W1})-(\ref{G-W3}) with $p\geq 2$ (but not necessarily the upper bound in \qref{p-bound}), 
then for $\vep$ sufficiently small and  $\eta_2< p \eta_1$ there exist constants $A_1, A_2>0$ depending  
only on $\eta_1,\eta_2, p , C_1, C_3$, and $C_4$, 
such that for any $u\in H_0^2(\Omega_\vep)$
	\begin{align}  
 	 	F_\vep|_{\Omega_\vep}(u)   &\geq \int_{\Omega_\vep}\left\{ 
    		\frac1{4\vep}\left( \frac{\delta E_\vep}{\delta u}\right)^2 
   		+ \frac{\eta_1\vep}2 |\nabla u|^2    + \frac{A_1}\vep |u|^p   \right\}\;dx - \frac{A_2}\vep |\Omega_\vep|.
   		\label{F-lowerbound}
	\end{align}
	
\end{lemma}

\begin{proof}
The $L^2$ variational derivative of $E$ takes the form 
\begin{align} \label{delta-cE}
\frac{\delta E_\vep}{\delta u} = -\vep\Delta u  + \frac1\vep W'(u),
\end{align}
and we may write
\begin{align}
   \int_{\Omega_\vep} \frac{\delta E_\vep}{\delta u} u\;dx 
   &=\int_{\Omega_\vep} \vep |\nabla u|^2 + \frac1\vep W'(u)u.
\end{align}
Then
\begin{align} \label{F-lowerbound-0}
   F_\vep|_{\Omega_\vep}(u)&=\int_{\Omega_\vep}\left\{ 
   \frac{1}{2\vep}\left(\frac{\delta E_\vep}{\delta u}\right)^2
   -\frac{\eta_1\vep}{2}|\nabla u|^2
   -\frac{\eta_2}\vep W(u)
   \right\}\;dx \nn\\
   &=\int_{\Omega_\vep}\Biggl\{ \frac1{2\vep}\left( \frac{\delta E_\vep}{\delta u}\right)^2
      - \eta_1 \frac{\delta E_\vep}{\delta u}u +\eta_1\left(\vep |\nabla u|^2  + \frac1\vep W'(u)u \right)
    \nn\\
   &\hspace{.5in} -\frac{\eta_1\vep}{2}|\nabla_s u|^2
   -\frac{\eta_2}\vep W(u) \Biggr\}\;dx.
\end{align}
Since 
\begin{align}
   \eta_1\frac{\delta E_\vep}{\delta u}u 
   \leq \frac{1}{4\eps} \left( \frac{\delta E_\vep}{\delta u}\right)^2  + \eps \eta_1^2 u^2,
\end{align}
plugging into \qref{F-lowerbound-0}, we have 
\begin{align}
   F_\vep|_{\Omega_\vep}(u) &\geq \int_{\Omega_\vep}\left\{ 
   \frac1{4\vep}\left( \frac{\delta E_\vep}{\delta u}\right)^2 
   + \frac{\eta_1\vep}2 |\nabla u|^2  \right.\nn\\
  &\hspace{.5 in}  \left.  + \frac1\vep\biggl( \eta_1 W'(u)u 
  - \eta_2W(u)- \eta_1^2\vep^2 u^2\biggr)    \right\}\;dx. \label{F-lowerbound1}
\end{align}
Since $p\geq 2$, by \qref{G-W1} and \qref{G-W3}, we have
\[ \eta_1 W'(u)u   - \eta_2W(u)- \eta_1^2\vep^2 u^2 \geq \left(C_1(\eta_1 p -\eta_2) -\vep^2\eta_1^2\right)|u|^p 
+ (\eta_1C_4 -\eta_2 C_3).\]
For $\vep$ sufficiently small and $\eta_2<\eta_1 p$ we find $A_1, A_2>0$, depending only on $\eta_1,\eta_2,$ 
$ p, C_1, C_3, C_4$ such that 
\[ \eta_1 W'(u)u   - \eta_2W(u)- \eta_1^2\vep^2 u^2 \geq A_1 |u|^p -A_2,\]
for all $u$. From the lower bound \qref{F-lowerbound1} we arrive at the estimate (\ref{F-lowerbound}).
\end{proof}
\vskip 0.1in
\noindent To establish the existence of a minimizer of $F_\vep$ we impose additional restrictions on $p$.
\begin{theorem} In addition to the assumptions in Lemma \ref{Lemma-lowerbound}, assume the upper 
bound on $p$ described in \qref{p-bound} holds, then there exists $u_\vep\in \cA_\vep$ that minimizes 
 $F_\vep|_{\Omega_\vep}$ over $\cA_\vep$.
\end{theorem}
\begin{proof}
	This theorem can be proved following a standard procedure (see, e.g., \cite{PZ-13}).  
	For the convenience of the readers, we briefly describe the procedures for $n\geq 3$. 
	The case $n=2$ is similar and simpler. Write $2^*:=2n/(n-2)$. Suppose 	
	$\{u_k\}$ is a minimizing sequence in $\cA_\vep$ for $F_\vep$. By \qref{F-lowerbound}, 
	$\{u_k\}$ is bounded in $H^1(\Omega_\vep)$. By the Sobolev embedding theorem, 
	$\{u_k\}$ is bounded in $L^q(\Omega_\vep)$ for any $1\leq q \leq 2^*$. Furthermore, by the 
	compact embedding theorem, 
	there is a subsequence, not relabeled, and a function $u_\vep\in H^1(\Omega_\vep)$ 
	such that 
	\begin{align} \label{sec2-conv1}
		u_k\to u_\vep \mbox{ a.e. in } \Omega_\vep \mbox{ and strongly in } L^q(\Omega_\vep)\mbox{ for }
		1\leq q< 2^*. 
	\end{align}
	If $p<\frac{2n-2}{n-2}$, then $|W'(u_k)|\sim |u_k|^{p-1}$ is bounded in $L^2(\Omega_\vep)$. 	
	By \qref{F-lowerbound}, $-\vep\Delta u_k +\frac1\vep W'(u_k)=\frac{\delta E_\vep}{\partial u}(u_k)$
	is bounded in $L^2(\Omega_\vep)$.  The triangle inequality implies that $\Delta u_k$ is bounded in 
	$L^2(\Omega_k)$.
	So actually $u_k$ is bounded in $H^2(\Omega_\vep)$. We can extract a further subsequence, not relabeled, 
	such that 
	\begin{align}
		u_k &\rightharpoonup u_\vep \mbox{ weakly in }H^2(\Omega_\vep),  \label{sec2-conv2} \\
		u_k & \to u_\vep \mbox{ strongly in } H^1(\Omega_\vep), \label{sec2-conv3} \\
		u_k & \to u_\vep  \mbox{ a.e. in }\Omega_\vep \mbox{ and strongly in } L^q(\Omega_\vep) 
			\label{sec2-conv4}
	\end{align}
	for any $1\leq q<\infty$ if $n\leq 4$ and $1\leq q <2n/(n-4)$ if $n>4$. Since $\cA_\vep$ is a closed 
	convex subset of $H^2(\Omega_\vep)$, 
	we see that $u_\vep\in\cA_\vep$.

	Since $u_k\to u_\vep$ a.e. in $\Omega_\vep$, by the continuity of $W'$, 
	we have $W'(u_k)\to W'(u_\vep)$ a.e. in $\Omega_\vep$. By the growth condition \qref{G-W2} of $W'$, 
	we have $|W'(u_k)| \leq C_1p|u_k|^{p-1}+ C_3'$. Since $|u_k|^{p-1}\to |u_\vep|^{p-1}$ strongly in
	$L^2(\Omega_\vep)$, 
	using the generalized Dominated Convergence Theorem, we know that $W'(u_k)\to W'(u_\vep)$ strongly in 
	$L^2(\Omega_\vep)$. 
	Combined with \qref{sec2-conv2}, we have
	\begin{align} 
		-\vep\Delta u_k + \frac1\vep W'(u_k) \rightharpoonup -\vep \Delta u_\vep + \frac1\vep W'(u_\vep) 
		\mbox{ weakly in } L^2(\Omega_\vep).
	\end{align}
	Hence 
	\begin{align}\label{sec2-term1}
		\int_{\Omega_\vep} \left(-\vep \Delta u_\vep + \frac1\vep W'(u_\vep) \right)^2\,dx 
		\leq \liminf_{k\to\infty} \int_{\Omega_\vep} \left( -\vep\Delta u_k + \frac1\vep W'(u_k)\right)^2\,dx. 
	\end{align}
	By \qref{sec2-conv3}, we have 
	\begin{align}\label{sec2-term2} 
		\int_{\Omega_\vep} |\nabla u_\vep|^2\,dx = \lim_{k\to\infty} \int_{\Omega_\vep} |\nabla u_k|^2\,dx. 
	\end{align}
	Since $|W(u)|\sim |u|^p$  as $|u|\to\infty$ and $2\leq p< \frac{2n-2}{n-2}$, which is smaller than 
	$2n/(n-4)$ if $n>4$, 
	by \qref{sec2-conv4} and the generalized Dominated 
	Convergence Theorem, we have 
	\begin{align}\label{sec2-term3}
		\int_{\Omega_\vep} W(u_\vep)\,dx = \lim_{k\to\infty} \int_{\Omega_\vep} W(u_k)\,dx. 
	\end{align}
	Combining \qref{sec2-term1}, \qref{sec2-term2}, and \qref{sec2-term3}, we have 
	$F_\vep|_{\Omega_\vep}(u_\vep) \leq \liminf_{k\to\infty}F_\vep|_{\Omega_\vep}(u_k)$.
	Since $u_k\in\cA_\vep$ is a minimizing sequence for $F_\vep|_{\Omega_\vep}(u_\vep)$ 
	and $u_\vep\in\cA_\vep$, 
	we conclude that 
	$u_\vep$ is a minimizer for $F_\vep|_{\Omega_\vep}$ in $\cA_\vep$.
\end{proof}
\vskip 0.1in

Since $F_\vep|_{\Omega_\vep}$ is a non-convex functional, the family of critical points is typically not 
unique. 
Since $\cA_\vep$
is a closed and convex subset of $H_0^2(\Omega)$, we may resort to techniques of variational inequalities. 
Let $u_\vep\in\cA_\vep$ be any minimizer of $F_\vep|_{\Omega_\vep}$ over $\cA_\vep$. Fix
any $v\in\cA_\vep$, define $j_{\Omega_\vep}(s):=F_\vep|_{\Omega_\vep}(u_\vep + s(v-u_\vep))$ for any 
$s\in [0,1]$. Then  $j_{\Omega_\vep}(0)\leq j_{\Omega_\vep}(s)$ for any $s\in [0,1]$. So 
if $j_{\Omega_\vep}(s)$ is
differentiable, then  $ j_{\Omega_\vep}'(0)\geq 0$ , which gives a variational inequality
\begin{align}\label{var-ineq1}
	&\left\langle \frac{\delta F_\vep|_{\Omega_\vep}}{\delta u}(u_\vep), v-u_\vep \right\rangle \nn\\
	&:=
	\int_{\Omega_\vep}  \biggl\{\frac1\vep\left( -\vep\Delta u_\vep +\frac1\vep W'(u_\vep)\right)
	\left( -\vep\Delta(v-u_\vep) 
	+\frac1\vep W''(u_\vep)(v-u_\vep)\right)  \nn\\
	& \qquad-  \left(\eta_1\vep \nabla u_\vep\cdot \nabla(v-u_\vep) + \frac{\eta_2}{\vep} W'(u_\vep) 
	(v-u_\vep) \right)
	\biggr\}dx \nn\\
	&\geq 0 \quad \mbox{ for all } v\in\cA_\vep.
\end{align}
Given a minimizer  $u_\vep\in \cA_\vep$ for $F_\vep|_{\Omega_\vep}$, we extend $u_\vep$ by zero 
outside of $\Omega_\vep$.  We denote this extension by $\overline u_\vep$. Then 
\[ \overline u_\vep \in \cA_{\Omega,\vep}:= \left\{ v\in H_0^2(\Omega), \quad v \geq 0 \mbox{ in } \Omega, 
	\quad \int_\Omega v \,dx = m_\vep\right\}.\]
$\overline u_\vep$ may not be a minimizer of $F_\vep$ over $\cA_{\Omega,\vep}$, but $\overline u_\vep$ does have some 
nice properties. 
We split 
$F_\vep$ into two parts, $F_\vep|_{\Omega_\vep}$ and $F_\vep|_{\Omega\setminus\overline\Omega_\vep}$.
For any $\phi\in H_0^2(\Omega\setminus\overline\Omega_\vep)$,  
a straightforward calculation shows that under the assumption \qref{A-W1}, the 
first order variational derivative of $F_\vep|_{\Omega\setminus\overline\Omega_\vep}$ at 0 along $\phi$ is 
zero, that is, 
\begin{align}\label{var-eq}	
	\left\langle\frac{\delta F_\vep|_{\Omega\setminus\overline\Omega_\vep}}{\delta u}(0), \phi \right\rangle 
	:= \lim_{s\to 0} \frac{F_\vep|_{\Omega\setminus\overline\Omega_\vep}(s\phi) 
	- F_\vep|_{\Omega\setminus\overline\Omega_\vep}(0)}{s}=0. 
\end{align}
The combination of \qref{var-ineq1} and \qref{var-eq} gives that for any $v\in H_0^2(\Omega)$ such that 
$v|_{\Omega_\vep} \in\cA_\vep$, we have
	\begin{align}\label{var-ineq2}
		\left\langle \frac{\delta F_\vep}{\delta u}(\overline u_\vep), v-\overline u_\vep \right\rangle \geq 0.
	\end{align}
	Here the variational derivative of $F_\vep$ is defined the same way as that of $F_\vep|_{\Omega_\vep}$ 
	in \qref{var-ineq1}, 
	with $\Omega_\vep$ replaced by $\Omega$. 
\qref{var-ineq2} states that $\overline u_\vep$ is a critical point of $F_\vep$ under perturbations that are 
away from 
$\partial\Omega_\vep$, and also preserves the total concentration in $\Omega_\vep$.  
In this sense, we say that $\overline u_\vep$ is a geometrically localized minimizer of $F_\vep$. 
If $u$ is a minimizer of $F_\vep$ over $\cA_{\Omega,\vep}$, 
 then $u$ satisfies 
\begin{align}\label{var-ineq3}
		\left\langle \frac{\delta F_\vep}{\delta u}( u), v-u \right\rangle \geq 0 \qquad\mbox{ for all }v\in\cA_{\Omega,\vep}.
\end{align}
In this case, the lipid domain $\Omega_\vep$ is implicitly defined as $\{x\in\Omega: u(x)>0\}$. 
We leave the exploration of \qref{var-ineq3} for future studies.

\section{Codimension one minimizers}

We address the issue of convergence of sequences with bounded energy as $\vep\to 0$ by imposing 
the additional  assumption that $\Omega_\vep$ is a thin region composed of points $x\in\Omega$
that are $O(\vep)$  from a sufficiently smooth, non-self-intersecting, codimension one hypersurface $\Gamma$. 
In parametric form,  we can write
$\Gamma=\{\phi(s): s=(s_1,\dots,s_{n-1})\in Q\subset \R^{n-1}\}$, where the  parameterization is chosen so that  
$s_i$ is the arc length along the $i^{\rm th}$ coordinate curve and the coordinate curves are lines of curvature. 
Let $\bn$ be the normal of $\Gamma$ and assume all principal curvatures $\kgam_j(s), j=1,\dots,n-1$, of 
$\Gamma$ are 
continuously differentiable and  bounded. That is, there  exists $\kgam_0>0$ such that 
\begin{equation}
\label{Gamma-bounds}
 |\kgam_j(s)|\leq \kgam_0, \quad \left| \frac{\partial \kgam_j}{\partial s_i}(s)\right|\leq \kgam_0 \quad\mbox{ for all }s\in Q 
	\mbox{ and all } i, j=1\dots,n-1.
\end{equation}
The boundedness of $\kgam_j(s)$ guarantees that there exists $\ell>0$ such that the thin region  
\begin{align} 
	\Omega_1^\ell:=\{\phi(s)+ z \bn(s): s\in Q, -\ell<z<\ell\}.
\end{align}
does not self-intersect.  The change of variables $x\mapsto (s,z)$ for $x\in\Omega_\vep$ given by 
\[ x=\phi(s) + \vep z\bn,\]
is well defined and smooth with a smooth inverse $x\mapsto (s,z)$ on $\Omega_1^\ell$. The inverse 
$x\mapsto \phi(s(x))$ 
is the projection of point 
$x\in\Omega_\vep$ onto $\Gamma$, and $z\in(-\ell, \ell)$ is the $\vep$-scaled signed distance of 
$x$ to $\Gamma$. 

We define  $\Omega_\vep\subset \Omega$ to be the region of distance $2\vep \ell$ to $\Gamma$. More
specifically
\beq
\Omega_\vep:=\{\phi(s) + \zeta\bn(s): s\in Q,  -\vep \ell<\zeta<\vep \ell\},
\eeq 
and it has volume $|\Omega_\vep| =\vep |\Omega_1^\ell|$. Defining 
$\tu(s,z) = u(x(s,z))$, then $\tu\in H^2(\Omega_1^\ell)$, and $\tu(s,\pm \ell)=0$ for all $s\in Q$.

We will use a subscript $x$ to indicate  operators in the Cartesian coordinates $x$. Let $k_1,\dots,k_{n-1}$
be the principal curvatures of $\Gamma$ and $\bT_1, \dots,\bT_{n-1}$ the corresponding unit tangent vectors.
 Under the change of variables
$x\mapsto (s,z)$,  
$\nabla_x u(x)$ and $\Delta_x u(x)$ have the following forms in the $(s,z)$ coordinates 
\cite{dp:bilayer, promislow-2011}
\begin{align}
   \nabla_x u(x) &=\sum_{j=1}^{n-1}\frac{\bT_j}{1+\vep z \kgam_j}\frac{\partial\tu}{\partial s_j} 
   + \vep^{-1}\bn\frac{\partial\tu}{\partial z} =:\tilde D\tu(s,z), \label{def-tildeD-tu} \\
\Delta_x u(x) &=\sum_{j=1}^{n-1}\frac{1}{(1+\vep z\kgam_j)^2}\frac{\partial^2 \tu}{\partial s_j^2} 
  + \vep^{-1} \sum_{j=1}^{n-1} \frac{\kgam_j}{1+\vep z\kgam_j}\frac{\partial\tu}{\partial z} 
 + \vep^{-2}\frac{\partial^2\tu}{\partial z^2}  \nn\\
    &\quad -\vep \sum_{j=1}^{n-1} \frac{\partial \kgam_j}{\partial s_j}\frac{1}{(1+\vep z \kgam_j)^3} 
    \frac{\partial \tu}{\partial s_j}
=:\tilde\Delta\tu(s,z). 
\label{def-tildeDelta-tu}
\end{align}
By \qref{def-tildeD-tu} and \qref{def-tildeDelta-tu}, $F_\vep(u) = \tF_\vep^\ell(\tu)$
where  
\begin{align} \label{def-tF^ell} 
\tF_\vep^\ell(\tu) &:= \int_{\Omega_1^\ell}\left\{\frac1{2}\left( -\vep\tilde\Delta \tu+\frac1\vep W'(\tu)\right)^2 
      - \left(\frac{\eta_1\vep^2}{2}|\tilde D \tu|^2 + {\eta_2}W(\tu) \right)  \right\}Jdsdz.
\end{align}
Here the scaled Jacobian can be expressed as
\beq
J(s,z) = \sum\limits_{j=0}^n \vep^jK_j(s)z^j,
\label{def-Jac}
\eeq
in terms of the j'th Gaussian curvatures $K_0=1$, and 
$$K_j:= \sum_{i_1<\cdots<i_j} \kappa_{j_1}\cdots\kappa_{i_j}.$$
In particular we remark that $K_1=H_0$ and $J\to1$ in all Sobolev norms as $\vep\to0$. 
In the scaled variables the lower bound \qref{F-lowerbound}  takes the form,  
\begin{align}      
      \tF_\vep^\ell(\tu) &\geq \int_{\Omega_1^\ell}\left\{\frac1{4}\left( -\vep\tilde\Delta \tu+\frac1\vep W'(\tu)\right)^2 
      + \frac{\eta_1\vep^2}{2}|\tilde D \tu|^2 + A_1 |\tu|^p \right\}Jdsdz - A_2 |\Omega_1^\ell|.\label{tF^b-lowerbound}
\end{align}
Taking $m_\vep=m\vep$ in the mass constraint, (\ref{mvep-def}), we rewrite this the equivalent condition,
 \beq
 \int_{\Omega_1^\ell} \tu \;J dsdz = m
 \eeq
 on the rescaled domain.
\begin{theorem}\label{sec3-T1}
Fix a codimension one interface $\Gamma$ with curvature bound $\kgam_0$ as above, and choose $\ell\in(0,1/(2\kgam_0))$. 
Then there exists $C>0$ such that for any codimension-one $\ell$-bounded sequence $\{\tu_k\}_{k=1}^\infty$ from $H^2(\Omega_1^\ell)$
the following bounds hold for all $k\in\N$, 
\beq \label{T1-B1}
\left\| \tu_k \right\|_{L^p(\Omega_1^\ell)} + 
	\left\| \frac{\partial \tu_k}{\partial z}\right\|_{L^2(\Omega_1^\ell)}  \!\! +
	\vep_k \sum\limits_{j=1}^{n-1} \left\| \frac{\partial \tu_k}{\partial s_j}\right\|_{L^2(\Omega_1^\ell)}
	\leq C.
\eeq
If we assume in addition that 
\beq\label{T1-B2}
 \sum\limits_{j=1}^{n-1} \left( \left\| \frac{\partial \tu_k}{\partial s_j}\right\|_{L^2(\Omega_1^\ell)} \!\!
+\vep_k \left\| \frac{\partial^2\tu_k}{\partial s_j^2}\right\|_{L^2(\Omega_1^\ell)}\right)  \leq C
\eeq
for all $k\in\N$,	then we have the following conclusions.
\begin{enumerate}
	\item In the limit when $k\to\infty$, the Cahn-Hilliard energy is equi-partitioned along the 
	normal direction, that is, 
	\begin{align} \label{equi-partition}
		\lim_{k\to\infty} \int_{\Omega_1^\ell} \left| \frac12 \left(\frac{\partial \tu_k}{\partial z} \right)^2 - W(\tu_k)
		\right|dsdz =0.
		\end{align}
	\item There exists a subsequence  $\tu_{k_i}$  and a $\tu^*\in H_0^1(\Omega_1^\ell)\cap L^{p}(\Omega_1^\ell)$, 
such that $\tu_{k_i}$ converges weakly to $\tu^*$. Moreover $\tu^*$ is a weak solution to the bilayer equation
		\beq \label{T1-BL}
 			-\frac{\partial^2\tu^*}{\partial z^2} + W'(\tu^*)=0,
 		\eeq
 	over $\Omega_1^\ell$, in the sense that 
 	\begin{align}
	\label{weak-BL}
		\int_{\Omega_1^\ell} \left(\frac{\partial\tu^*}{\partial z} \frac{\partial\phi}{\partial z}+ W'(\tu^*)\phi\right)
			\,dsdz =0,
	\end{align}
	for all $\phi\in H_0^1(\Omega_1^\ell)$ if $p\leq2^*=2n/(n-2)$, and for all 
	$\phi\in H_0^1(\Omega_1^\ell)\cap L^p(\Omega_1^\ell)$ if $p>2^*.$
 \end{enumerate}
\end{theorem}
\vskip 0.1in
\noindent
\begin{proof}
We infer from \qref{tF^b-lowerbound}
that that there exists $M>0$ such that
\begin{align}\label{upper-bound0}
	\int_{\Omega_1^\ell}\left\{\frac1{4}\left( -\vep_k\tilde\Delta \tu_k+\frac1{\vep_k} W'(\tu_k)\right)^2 
      + \frac{\eta_1\vep_k^2}{2}|\tilde D \tu_k|^2 + A_1  |\tu_k|^p \right\}dsdz \leq M.
\end{align}
Dropping the first two terms on  the left-hand side of (\ref{upper-bound0}) yields the bound  
	\begin{align}
		\left\| \tu_k \right\|_{L^p(\Omega_1^\ell)} 
		& \leq \left(\frac{M}{A_1}\right)^{1/p}. \label{u-bound}
	\end{align}
Recalling \qref{def-tildeD-tu} and keeping only the second term on the left-hand side of (\ref{upper-bound0}) implies the bound
\begin{align}
	\frac{\eta_1}{2}\int_{\Omega_1^\ell} \left\{\sum_{j=1}^{n-1}\left(\frac{\vep_k}{1+\vep_k z \kgam_j} \right)^2 
	\left|\frac{\partial\tu_k}{\partial s_j} \right|^2 
	+ \left| \frac{\partial \tu_k}{\partial z}\right|^2\right\}J dsdz \leq M.
\end{align}
Since $|\kgam_j|\leq \kgam_0$ and $|z|\leq  \ell$ on $\Omega_1^\ell$, we have the uniform estimate
$$ \frac{2\vep_k}{3}\leq \frac{\vep_k}{1 +\vep_k z \kgam_j} \leq 2\vep_k,$$
for all $k$ such that $\vep_k<1/(2l\kgam_0)$, 
and hence  $\frac{\partial\tu_k}{\partial z}$ is uniformly bounded while $\frac{\partial\tu_k}{\partial s_j}$
can grow at most as fast $\vep_k^{-1}.$ More specifically we have established that
\begin{align}
		\disp \left\| \frac{\partial \tu_k}{\partial s_j}\right\|_{L^2(\Omega_1^\ell)} 
		& \leq \frac{3\sqrt{M}}{\sqrt{2\eta_1}\vep_k}, \label{Ds-u-bound}\\
		\disp \left\| \frac{\partial \tu_k}{\partial z}\right\|_{L^2(\Omega_1^\ell)}  
		&\leq \frac{\sqrt{2M}}{\sqrt{\eta_1}}. \label{Dz-u-bound}
\end{align}
Taken together, (\ref{u-bound}), (\ref{Ds-u-bound}), and (\ref{Dz-u-bound}) imply (\ref{T1-B1}).

To establish the convergence to a weak solution of the bilayer equation, we impose the enhanced constraint 
and (\ref{T1-B2}) return to the first term on the left-hand side of (\ref{upper-bound0}), deducing that
	\begin{align}      
     	\left\| -\vep_k\tilde\Delta\tu_k + \frac1{\vep_k} W'(\tu_k)  
	\right\|_{L^2(\Omega_1^\ell)} \leq 2\sqrt{M}.
	\label{term1-bound}
	\end{align}
Using \qref{def-tildeDelta-tu} to expand the left-hand side of (\ref{term1-bound}), 
we group terms by formal powers of $\vep$ , 
	\begin{align*}
		-\vep_k\tilde\Delta\tu_k + \frac1{\vep_k} W'(\tu_k) 
		&= -\frac1{\vep_k}\frac{\partial^2\tu_k}{\partial z^2} + \frac1{\vep_k} W'(\tu_k) 
     	-\sum_{j=1}^{n-1} \frac{\kgam_j}{1+\vep_k z\kgam_j}\frac{\partial\tu_k}{\partial z}  \nn\\
	&\qquad-\vep_k \sum_{j=1}^{n-1}\frac{1}{(1+\vep_k z\kgam_j)^2}\frac{\partial^2 \tu_k}{\partial s_j^2} 
	+\vep_k^2 \sum_{j=1}^{n-1} \frac{\partial \kgam_j}{\partial s_j}\frac{1}{(1+\vep_k z \kgam_j)^3} 
	\frac{\partial \tu_k}{\partial s_j},
	\end{align*}
and employ the triangle inequality
	\begin{align}
		&\left\| \frac1{\vep_k}\left(-\frac{\partial^2\tu_k}{\partial z^2} + W'(\tu_k)\right)  
		-\vep_k \sum_{j=1}^{n-1}\frac{1}{(1+\vep_k z\kgam_j)^2}\frac{\partial^2 \tu_k}{\partial s_j^2}
		\right\|_{L^2(\Omega_1^\ell)} \nn\\
		\leq& \left\| -\vep_k\tilde\Delta\tu_k + \frac1{\vep_k} W'(\tu_k)  \right\|_{L^2(\Omega_1^\ell)} 
			+ \sum_{j=1}^{n-1} \left\| \frac{\kgam_j}{1+\vep_k z\kgam_j} \right\|_{L^\infty(\Omega_1^\ell)}
			\left\| \frac{\partial\tu_k}{\partial z} \right\|_{L^2(\Omega_1^\ell)} \nn\\
		& + \vep_k^2 \sum_{j=1}^{n-1}\left\| \frac{\partial \kgam_j}{\partial s_j} \frac{1}{(1+\vep_k z \kgam_j)^3}
			\right\|_{L^\infty(\Omega_1^\ell)}
			\left\| \frac{\partial \tu_k}{\partial s_j} \right\|_{L^2(\Omega_1^\ell}.
	\end{align}
Combining the uniform bounds on $\vep_k$ over $k\geq 1$ and on $\kgam_j$ and $\frac{\partial \kgam_j}{\partial s_j}$ for $j=1,\ldots, n-1$, with the bounds 
\qref{Ds-u-bound}, \qref{Dz-u-bound}, and \qref{term1-bound}, imply the existence of a constant $M_1>0$
such that for all $k\in\N$
	\begin{align}
		&\left\| \frac1{\vep_k}\left(-\frac{\partial^2\tu_k}{\partial z^2} + W'(\tu_k)\right)  
		-\vep_k \sum_{j=1}^{n-1}\frac{1}{(1+\vep_k z\kgam_j)^2}\frac{\partial^2 \tu_k}{\partial s_j^2}
		\right\|_{L^2(\Omega_1^\ell)} \leq M_1.
	\end{align}
The assumption (\ref{T1-B2}) implies that the tangential second derivatives scale as $O(\vep_k^{-1})$ and are lower order. Moving
them to the right-hand side we conclude that
	\begin{align}
		\left\| \frac1{\vep_k}\left(-\frac{\partial^2\tu_k}{\partial z^2} + W'(\tu_k)\right)  
		\right\|_{L^2(\Omega_1^\ell)} 
		&\leq M_1 + \vep_k \sum_{j=1}^{n-1}\left\|\frac{1}{(1+\vep_k z\kgam_j)^2}\right\|_{L^\infty(\Omega_1^\ell)}
		\left\|\frac{\partial^2 \tu_k}{\partial s_j^2}
		\right\|_{L^2(\Omega_1^\ell)} \nn\\
		&\leq M_2,
	\end{align}
for some constant $M_2>0$ independent of $\vep,$ and consequently 
	\begin{align} \label{conv-1}
		\left\| -\frac{\partial^2\tu_k}{\partial z^2} + W'(\tu_k)  \right\|_{L^2(\Omega_1^\ell)} \leq M_2\vep_k\to 0
	\end{align}
	as $k\to\infty$.
We use this convergence to establish the equi-partition, (\ref{equi-partition}) (i). Since all terms are zero at $z=\pm \ell$ we have the bound
\begin{align*}
	&\left| \frac12 \left(\frac{\partial \tu_k}{\partial z}(z) \right)^2 - W(\tu_k(z)) \right| 
	=\left| \int_{-l}^z \left( \frac{\partial^2\tu_k}{\partial z^2}(\zeta) - W'(\tu_k(\zeta))\right)
	\frac{\partial\tu_k}{\partial z}(\zeta)\,d\zeta\right| \nn\\
	&\leq \left(\int_{-l}^l \left( \frac{\partial^2\tu_k}{\partial z^2}(\zeta) - W'(\tu_k(\zeta)\right)^2d\zeta\right)^{1/2}
	\left( \int_{-l}^l \left|\frac{\partial\tu_k}{\partial z}(\zeta)\right|^2d\zeta\right)^{1/2},
\end{align*}
from which we deduce the $L^1$ estimate
\begin{align*}
	&\int_{\Omega_1^\ell} \left| \frac12 \left(\frac{\partial \tu_k}{\partial z}(z) \right)^2 - W(\tu_k(z)) \right| dsdz,
	\nn\\
	&\leq \left( \int_{\Omega_1^l}\int_{-l}^l\left( \frac{\partial^2\tu_k}{\partial z^2}(\zeta) - W'(\tu_k(\zeta))\right)^2
	d\zeta dsdz\right)^{1/2}\!\!
	\left( \int_{\Omega_1^\ell}\int_{-l}^l \left|\frac{\partial\tu_k}{\partial z}(\zeta)\right|^2d\zeta dsdz\right)^{1/2} \!\!,
	\nn\\
	&\leq 2l \left\| -\frac{\partial^2\tu_k}{\partial z^2} + W'(\tu_k)  \right\|_{L^2(\Omega_1^\ell)} 
	\left\| \frac{\partial\tu_k}{\partial z} \right\|_{L^2(\Omega_1^\ell)}. \end{align*}
From \qref{T1-B1} and \qref{conv-1} we see that the right-hand side tends to zero as $k\to\infty$ and we deduce (i).

To establish that the limit function $\tu_*$ is a weak solution of the bilayer equation, (ii)-\qref{T1-BL}, we recall that assumption (\ref{T1-B2}) implies that the first 
order tangential derivatives  $\frac{\partial \tu_k}{\partial s_j}$ are uniformly bounded in $L^2(\Omega_1^\ell)$ for $j=1,\cdots, n-1$ and all $k\in\N.$  
This fact, in conjunction with \qref{u-bound} and \qref{Dz-u-bound}, implies that $\tu_k$ is bounded in $H^1(\Omega_1^\ell)$. We deduce that there exists 
a subsequence 
$\tu_{k_i}$ and a function $\tu^*\in H^1(\Omega_1^\ell)$ such that 
\begin{align} \label{conv-2}
		\tu_{k_i} \rightharpoonup \tu^* \quad\mbox{weakly in } H^1(\Omega_1^\ell).
\end{align}
Moreover the Sobolev embedding and compact embedding theorems imply that $\tu_k$ is bounded in $L^{2^*}(\Omega_1^\ell)$
and  $\tu_{k_i}\to \tu^*$ strongly in $L^q(\Omega_1^\ell)$ 	for any  $1\leq q < 2^*:=2n/(n-2)$. 
Extracting a further subsequence, not relabelled, 
we have that $\tu_{k_i}\to \tu^*$ a.e. in $\Omega_1^\ell$. Since $\tu_k$ is bounded in  $L^p(\Omega_1^\ell)$, 
we can improve the strong convergence so that 
	\begin{align}
		\tu_{k_i}\to \tu^*\quad\mbox{ strongly in } L^q(\Omega_1^\ell) 
		\mbox{ for any } 1\leq q <q_0:=\max\{p, 2n/(n-2)\}. 
	\end{align}
Proceeding, form the $H^1$ weak convergence, for any $\phi\in H_0^1(\Omega_1^\ell)$, we have
	\begin{align}
		\lim_{i\to\infty}\int_{\Omega_1^\ell} \frac{\partial\tu_{k_i}}{\partial z} \frac{\partial\phi}{\partial z}\,Jdsdz 
		&= \int_{\Omega_1^\ell} \frac{\partial\tu^*}{\partial z} \frac{\partial\phi}{\partial z}\,Jdsdz. 
		\label{weak-eq-part1}
	\end{align}
Since $W'(u)$ is continuous in $u$, and $|W'(u)|\leq C |u|^{p-1}$ as $|u|\to\infty$, the strong convergence of 
$\tu_{k_i}\to \tu^*$ in $L^{q}(\Omega_1^\ell)$  for any $q\in[p-1,q_0)$ 
and a.e. convergence in $\Omega_1^\ell$, together with the
Generalized Dominated Convergence Theorem imply that for any $q\in[p-1,q_0)$, 
\begin{align}\label{W'-conv-1}
	W'(\tu_{k_i}) \to W'(\tu^*) \quad\mbox{strongly in }L^{q/(p-1)}(\Omega_1^\ell)\mbox{ and a.e. in }
	\Omega_1^\ell.
\end{align}
Since $W'(\tu_{k_i})$ is bounded in $L^{q_0/(p-1)}(\Omega_1^\ell)$, by extracting a further 
subsequence if necessary, 
by \qref{W'-conv-1}, we obtain
\begin{align}\label{W'-conv-2}
	W'(\tu_{k_i}) \rightharpoonup W'(\tu^*) \quad\mbox{weakly in }L^{q_0/(p-1)}(\Omega_1^\ell).
\end{align}
Thus
\begin{align}
 \lim_{i\to\infty} \int_{\Omega_1^\ell}W'(\tu_{k_i})\phi\,dx = \int_{\Omega_1^\ell} W'(\tu^*)\phi\,dx,
\label{weak-eq-part2}
\end{align}
for any $\phi\in L^r(\Omega_1^\ell)$ with $r= q_0/(q_0-p+1)$. Combining \qref{weak-eq-part1} and 
\qref{weak-eq-part2}, we deduce that
	\begin{align}
		\int_{\Omega_1^\ell} \left(\frac{\partial\tu^*}{\partial z} \frac{\partial\phi}{\partial z}+ W'(\tu^*)\phi\right)
			\,Jdsdz 
		&= \lim_{i\to\infty} \int_{\Omega_1^\ell} \left(\frac{\partial\tu_{k_i}}{\partial z} \frac{\partial\phi}{\partial z} 
			+ W'(\tu_{k_i})\phi\right) \, Jdsdz\nn\\
		&= \lim_{i\to\infty} \int_{\Omega_1^\ell} \left(-\frac{\partial^2\tu_{k_i}}{\partial z^2}  
		+ W'(\tu_{k_i})\right)\phi \, Jdsdz \nn\\
		&=0,
	\end{align}
for any $\phi\in H_0^1(\Omega_1^\ell)\cap L^r(\Omega_1^\ell)$.
We remark that if $p\leq 2^*$, then $q_0 = 2^*$ and $r\leq 2^*$, so that the Sobolev embedding theorem implies
that $H_0^1(\Omega_1^\ell)\cap L^r(\Omega_1^\ell)= H_0^1(\Omega_1^\ell)$. On the other hand, if $p>2^*$, then $q_0=p$, $r=p$, 
and  $H_0^1(\Omega_1^\ell)\cap L^r(\Omega_1^\ell) = H_0^1(\Omega_1^\ell)\cap L^p(\Omega_1^\ell)$.

\end{proof}


\section{Upper and lower bounds on codimension-one sequences}
A key goal of our analysis is to identify properties of codimension-one $\ell$-bounded energy sequences with characterize the form of their limiting energy. Theorem 3.1 establishes a condition on the tangential derivatives which guarantees that such sequences have subsequences which converge to weak solutions of the bilayer equation. We show that the asymptotic scaling of the tangential derivatives is essential
to the form of the limiting energy.

\subsection{Codimension-one lower bounds} 
We establish that a slightly stronger constraint on the tangential derivatives  imposes a class of
codimension-one lower bounds for the energy of each subsequence that has an $H^1$ weak limit. 
More specifically if $\tu^*\in H^1(\Omega_1^\ell)$ is a weak solution of the bilayer equation in the sense of \qref{weak-BL}, then we 
define the associated codimension-one energy
\beq
\label{def-CD1-Energy}
	G_1(\Gamma,u^*):= \int_\Gamma \left(\tilde a^*H_0^2 -(\eta_1+\eta_2)\tilde b^*\right) ds,
\eeq
where $H_0:=\sum_{j=1}^{n-1} \kgam_j$ is the total curvature of $\Gamma$, and 
\begin{align} 
	\tilde a^*(s) &:= \frac12\int_{-\ell}^\ell  \left| \frac{\partial\tu^*}{\partial z}\right|^2\,dz, \label{def-a*}\\
	 \tilde b^*(s)&:=\int_{-\ell}^\ell W(\tu^*) dz. \label{def-b*}
\end{align}
While this result falls short of establishing a unique codimension-one limiting energy it establishes a liminf inequality for
a class of codimension one $\ell$-bounded sequences. 

\begin{theorem} \label{sec3-T2}
Suppose $p<2^*$ and let $\{\tu_k\}_{k=1}^\infty$ be a codimension-one $\ell$-bounded sequence satisfying the assumptions of 
Theorem \ref{sec3-T1}; in particular \qref{T1-B1} holds. If in addition we strengthen assumption \qref{T1-B2} to include a stronger bound 
on the second tangential derivatives,
\begin{align}
\label{T2-B3}
 \sum\limits_{j=1}^{n-1} \left\| \frac{\partial^2\tu_k}{\partial s_j^2}\right\|_{L^2(\Omega_1^\ell)} = o(\vep_k^{-1}),
\end{align}
then for any subsequence $\tu_{\kgam_j}$ that  converges weakly to a function $\tu^*$ in $H_0^1(\Omega_1^\ell)$, 
we have the following codimension-one liminf inequality:
\begin{align}\label{liminf-ineq}
	\liminf_{j\to\infty}\tF_{\vep_{\kgam_j}}^\ell(\tu_{\kgam_j}) \geq  G_1(\Gamma,\tu^*).
\end{align}
\end{theorem}

\begin{proof} 
Without loss of generality, we assume $\tu_k \rightharpoonup \tu^*$ in $H_0^1(\Omega_1^\ell)$ and $\tu_k$ satisfies all 
assumptions in the theorem. We write the energy integral as the difference of the quadratic and functionalization terms,
\begin{align}
	\tF_{\vep_k}^\ell(\tu_k) &= \int_{\Omega_1^\ell}\left\{\frac1{2}\left( -\vep_k\tilde\Delta \tu_k
	+\frac1{\vep_k} W'(\tu_k)\right)^2
      - \left(\frac{\eta_1\vep_k^2}{2}|\tilde D \tu_k|^2 + {\eta_2}W(\tu_k) \right) \right\}Jdsdz \nn\\
      &=\hspace{1.3in}I\hspace{0.65in} - \hspace{0.8in}I\!I.
\end{align}
Using the Laplacian expansion \qref{def-tildeDelta-tu} we rewrite $I$ as
\begin{align}
	I &= \int_{\Omega_1^\ell}\frac1{2}\left( -\vep_k\tilde\Delta \tu_k+\frac1{\vep_k} W'(\tu_k)\right)^2\,Jdsdz \nn\\
	&= \int_{\Omega_1^\ell}\frac12\left\{-\vep_k\sum_{j=1}^{n-1}\frac{1}{(1+\vep_k z\kgam_j)^2}
	\frac{\partial^2 \tu_k}{\partial s_j^2} 
  	- \sum_{j=1}^{n-1} \frac{\kgam_j}{1+\vep_k z\kgam_j}\frac{\partial\tu_k}{\partial z}  \right.\nn\\
 	&\quad\quad\left. +\vep_k^2 \sum_{j=1}^{n-1} \frac{\partial \kgam_j}{\partial s_j}\frac{1}{(1+\vep_k z \kgam_j)^3} 
	\frac{\partial \tu_k}{\partial s_j}
		+ \frac1{\vep_k}\left(-\frac{\partial^2\tu_k}{\partial z^2}+ W'(\tu_k)\right)\right\}^2\,Jdsdz.
\end{align}
From \qref{T1-B2} and \qref{T2-B3}, it is easy to see that tangential derivative
terms, the first and third terms in the quadratic expression, tend to zero in $L^2(\Omega_1^\ell)$ as $k\to\infty,$
while from From \qref{Dz-u-bound} and \qref{conv-1}, the through-plane derivative terms, the second and fourth terms,
are uniformly bounded in $L^2(\Omega_1^\ell).$ It follows that the tangential terms may be neglected in the $k\to\infty$ limit.
Moreover 
\[ \sum_{j=1}^{n-1} \frac{\kgam_j}{1+\vep_k z\kgam_j} \to H_0, \]
in $L^\infty(\Omega_1^\ell)$  as $k\to\infty$ and we may make this replacement in the limit, observing that
\begin{align}
	\lim\limits_{k\to\infty} I &=\lim\limits_{k\to\infty} \int_{\Omega_1^\ell} \left\{ \frac12 H_0^2 \left(\frac{\partial\tu_k}{\partial z} \right)^2 
	+ \frac1{2\vep_k^2}	\left( - \frac{\partial^2\tu_k}{\partial z^2} + W'(\tu_k)\right)^2 \right.\nn \\
	&\hspace{1.5in}\left. - \frac1{\vep_k}  H_0 \frac{\partial\tu_k}{\partial z} \left( - \frac{\partial^2\tu_k}{\partial z^2} + W'(\tu_k)\right) \right\}\,Jdsdz,\nn\\
	&=: I_1+I_2+I_3.
\end{align}
	

	

Addressing these terms one-by-one we find,
\begin{align}
	\lim_{k\to\infty} I_3 
	&=\lim_{k\to\infty} -\frac{1}{\vep_k} \int_{\Gamma} H_0 \int_{-\ell}^\ell 
	\frac{\partial\tu_k}{\partial z} \left( - \frac{\partial^2\tu_k}{\partial z^2} + W'(\tu_k)\right)J(s,z)dzds \nn\\
	&= \lim_{k\to\infty} -\frac{1}{\vep_k} \int_{\Gamma} H_0 \int_{-\ell}^\ell 
	\frac{\partial}{\partial z} \left( -\frac12\left| \frac{\partial \tu_k}{\partial z}\right|^2 + W(\tu_k)\right)J(s,z) dzds\nn\\
	&= \lim_{k\to\infty} \frac{1}{\vep_k} \int_{\Gamma} H_0 \int_{-\ell}^\ell 
	\left( -\frac12\left| \frac{\partial \tu_k}{\partial z}\right|^2 + W(\tu_k)\right)\partial_zJ(s,z) dzds,
\end{align}
where we used $W(\tu_k(\pm l))=W(0)=0$ and $\frac{\partial \tu_k}{\partial z}(\pm l) =0$ in the integration by parts step.
Recalling the form of the Jacobian, \qref{def-Jac}, and the definitions \qref{def-a*} and \qref{def-b*} we find that
\beq
\label{I3-lim}
 \lim_{k\to\infty} I_3 =\lim_{k\to\infty} \int_{\Gamma} H_0^2(s)
 \int_{-\ell}^\ell \left( -\frac12\left| \frac{\partial \tu_k}{\partial z}\right|^2 + W(\tu_k)\right)\,dzds = 0,
 \eeq
the last conclusion relies on equipartition, \qref{equi-partition}.
 

For $I_1$, since $\tu_k\rightharpoonup \tu^*$ 
weakly in 
$H^1(\Omega_1^\ell)$, we have
\begin{align}
	\sum_{j=1}^{n-1} \frac{\kgam_j}{1+\vep_k z\kgam_j}\frac{\partial\tu_k}{\partial z} \rightharpoonup 
	\frac{\partial\tu^*}{\partial z}\sum_{j=1}^{n-1} \kgam_j = \frac{\partial\tu^*}{\partial z} H_0
	\quad\mbox{weakly in } H^1(\Omega_1^\ell).
\end{align}
By weak lower semicontinuity and the strong convergence of $J$ as $\vep\to 0$ we deduce that
\begin{align} \label{liminf-1}
	\liminf_{k\to\infty} I_1 \geq \frac12\int_{\Omega_1^\ell} \left(\frac{\partial\tu^*}{\partial z} H_0 \right)^2 J\bigl|_{\vep=0}dsdz
	= \int_\Gamma \tilde a^*H_0^2ds . 
\end{align}
For $I_2$ we merely observe that it is positive and bounded below by zero, hence $\liminf_{k\to\infty} I_2\geq 0.$
For $I\!I$, we have 
\begin{align}
	 I\!I &= \!\!\int_{\Omega_1^\ell} \left\{
	  \frac{\eta_1}2\left|\frac{\partial\tu_k}{\partial z}\right|^2 
	+\eta_2 W(\tu_k) 
	  + \frac{\eta_1\vep_k^2}2\sum_{j=1}^{n-1} \frac{1}{(1+\vep_k z \kgam_j)^2} 
	\left|\frac{\partial\tu_k}{\partial s_j}\right|^2\right\}Jdsdz.
\end{align}
Since $p<2^*$, by the compact embedding theorem we have $\tu_k\to \tu^*$ strongly in $L^p(\Omega_1^\ell)$, 
and hence
$W(\tu_k)\to W(\tu^*)$ strongly in $L^1(\Omega_1^\ell)$.
From the strong convergence of the tangential derivatives, \qref{T1-B2},  we have
\[ \lim_{k\to\infty} \int_{\Omega_1^\ell} \frac{\eta_1\vep_k^2}2\sum_{j=1}^{n-1} \frac{1}{(1+\vep_k z \kgam_j)^2} 
	\left|\frac{\partial\tu_k}{\partial s_j}\right|^2\,dsdz =0.\] 
Finally equipartition, \qref{equi-partition}, allows us to deduce that
\begin{align}\label{liminf-2}
	\lim\limits_{k\to\infty} I\!I  &= \lim\limits_{k\to\infty} \int_{\Omega_1^\ell} \left\{
	\frac{\eta_1}2\left|\frac{\partial\tu_k}{\partial z}\right|^2 
	+\eta_2 W(\tu_k) \right\}\,Jdsdz, \nn\\
	&=(\eta_1+\eta_2) \lim_{k\to\infty}\int_{\Omega_1^\ell} \ W(\tu_k)\, Jdsdz, \nn\\
	&=(\eta_1+\eta_2) \int_{\Gamma} b^*\,ds.
\end{align}
Combining \qref{I3-lim}, \qref{liminf-1} and \qref{liminf-2} we obtain \qref{liminf-ineq}.
\end{proof}

In general $\tilde a^*$ and $\tilde b^*$ may depend on $s$. 
Since we have the limiting equipartition of energy, \qref{equi-partition}, ideally we would hope to 
keep the equipartition of energy in the limiting profile $\tu^*$, i.e., $\tilde a^* = \tilde b^*$. However, 
by the weak convergence $\tu_{k_j}\rightharpoonup \tilde u^*$ in
 $H_0^1(\Omega_1^\ell)$, and the strong convergence $W(\tu_{k_j})\to W(\tu^*)$ in $L^1(\Omega_1^\ell)$, 
 we can only obtain 
\[ 0\leq \int_\Gamma \tilde a^*(s)\,ds \leq \int_\Gamma \tilde b^*(s)\,ds.\]
Indeed we have the following string of inequalities
\begin{align}  
\int_\Gamma \tilde a^*(s)\,ds&=\int_{\Omega_1^\ell}  \left| \frac{\partial\tu^*}{\partial z} \right|^2 dzds 
 \leq \liminf_{j\to\infty} \int_{\Omega_1^\ell}  \left| \frac{\partial\tu_{k_j}}{\partial z} \right|^2 Jdzds \nn\\
 &=\liminf_{j\to\infty} \int_{\Omega_1^\ell}W(\tu_{k_j})\,J dzds 
 =\int_{\Omega_1^\ell} W(\tu^*)\,dzds = \int_\Gamma \tilde b^*(s)\,ds.\nn
 \end{align}
It is an interesting question to explore the possible loss of energy in the weak convergence. This is left for future 
studies. 

We define  $\cS(\Gamma,\ell)$ to consist of the set of $\tu\in H^1(\Omega_1^\ell)$ that are weak solutions
of the bilayer equation, and introduce the set $\overline{\{\tu_k\}}^{H^1}$ of functions $\tu\in H^1(\Omega_1^\ell)$ which
are $H^1$ weak limits of a subsequence of $\{\tu_k\}$. With this notation we reformulate Theorem 4.1.
\begin{corollary}
Fix a codimension-one interface $\Gamma$ satisfying \qref{Gamma-bounds}. Every codimension-one $\ell$-bounded sequence $\{\tu_k\}_{k=1}^\infty$ sequence that
satisfies \qref{T1-B2} and \qref{T2-B3} has the lower bound
\begin{align}
	\liminf_{k\to\infty} \tF_{\vep_k}^\ell(\tu_k) \geq \sup \left\{G_1(\Gamma,\tu^*)\,\bigl| \,\tu^*\in\cS(\Gamma,\ell)\cap \overline{\{\tu_k\}}^{H^1}\right\}.
\end{align}
\end{corollary}
\begin{proof}
By Theorem 3.1 the set $\cS(\Gamma,\ell)\cap \overline{\{\tu_k\}}^{H^1}$ is not empty and by Theorem 4.1 for each $\tu_*$ in this set, 
the value $G_1(\Gamma,\tu_*)$ is a lower bound for $\liminf \tF_{\vep_k}^\ell(\tu_k).$
\end{proof}


\subsection{Codimension-one upper bounds}

For a given codimension one interface, $\Gamma$, upper bounds on the limiting energy of codimension-one $\ell$-bounded sequences 
can be obtained for specific examples. We present a sequence which satisfies the enhanced bounds \qref{T1-B2} and \qref{T2-B3}, and
whose codimension one energy provides a sharp lower bound. We also present a codimension-one $\ell$-bounded sequence which 
does not satisfy the enhanced bounds, has no subsequences which are weakly convergent in $H^1$, yet nonetheless has a
limiting energy that may be higher or lower than the codimension-one energy, particularly if the curvatures of $\Gamma$ are sufficiently
large.

 \subsubsection{Sharp codimension-one energy}
In light of Corollary 3.3, to construct sharp bounds it seems meritorious to reduce the size of the set $\cS(\Gamma,\ell).$  If the $H^1$ weak
closure of the sequence is also a $H^2$ weak closure, then the limiting bilayer equation solution will reside in $H^2_0(\Omega_1^\ell).$ 
The additional regularity implies that $\tu^*$ is a strong solution of the bilayer equation and hence is comprised of $N$,
$s$-dependent curves of translates of the bilayer profile, $U_*^1\in H^2(\R)$, defined as the unique single-pulse solution of
\qref{profile-eq1} that is symmetric about $z=0.$ The profile $U_*^1$ has compact support, denoted $[-L, L]$, and remark 
that $N$ must satisfy $NL\leq \ell$.

Indeed, for $\tu^*\in \cS(\Gamma,\ell)\cap H^2_0(\Omega_1^\ell)$, then the extension 
\begin{align} \label{assumption-extension}
	\tu_{\rm ext}^* := \left\{ 
		\begin{array}{ll}
			\tu^* & \mbox{ if }x\in \Omega_1^\ell, \\
			0 &\mbox{ otherwise}
		\end{array}
		\right.
\end{align}
 belongs to $H^2(\R^n)$ and
and  $\tu^*$ is a strong solution of the bilayer equation
\begin{align} \label{profile-eq1}
	\disp  -\frac{ \partial^2\tu^*}{\partial z^2} + W'(\tu^*)=0,\qquad u^*(\pm l)=0,
	\qquad \frac{\partial\tu^*}{\partial z} (\pm l) =0,\
\end{align}
This equation has a unique solution	
on each whisker $W(s):=\{x\in\Omega_1^\ell \bigl | \, \phi(x)=s\}.$
By standard dynamical systems techniques we find that $\tu^*$ is a superposition
of at most $N$ compactly supported single-pulse bilayer solutions 
\beq
 \tu_* = \sum\limits_{k=1}^N U_*^1(z-p_k(s)), 
 \eeq
 where the translates $p_k:\Gamma\mapsto\R$ are sufficiently far apart that their supports are disjoint and avoid the boundary.
 More specifically this is achieved if require that  $p_k<p_{k+1}+2L$, $p_1>-\ell+L$, $p_N<\ell-L$, uniformly over $s\in\Gamma.$ 
 Since $u_*\in H^2$ we deduce
 that each $p_k\in H^2(\Gamma)$. The precise value of $N$ depends upon the choice of $\ell$ and 
 the value of  the mass constraint. Pulling $\tu_*$ back to its unscaled version $u_*$, we see that total mass can 
 only take discrete values at leading order
 \beq
 \int_{\Omega} u_* dx =\vep\int_{\Omega_1^\ell} \tu_* \,Jdsdz= \eps |\Gamma| N \int_{-L}^L U_*^1(z)\, dz +O(\vep^2).
 \eeq
 Due to the rescaling, the translates have no impact on the mass in the limit as $\vep\to0.$

For fixed $\Gamma$ and $\ell$ sufficiently large we may tune value of the total mass, so that $N=1$. 
The $H^2$ solution set then reduces to the translates of the single pulse, which we denote by $\tu^1,$
and for any sequence $\{\vep_k\}$ we construct the corresponding sequence $\{\tu_k\}$ which agree trivially
with $\tu^1$ for each $k$, and their un-scaled forms $\{u_k\}.$ This is a codimension-one $\ell$-bounded sequence
that satisfies \qref{T1-B2} and \qref{T2-B3}. To evaluate its energy we rewrite (\ref{def-tF^ell}) as
\begin{align} \label{kF} 
\tF_\vep^\ell (\tu_k) &:= \int_{\Omega_1^\ell}\Bigl[ \frac{1}{2}\left( \vep^{-1}\left(-\partial_z^2 \tu^1+W'(\tu^1)\right)-H\partial_z \tu^1  
	-\vep \Delta_s \tu^!\right)^2 +\nonumber\\
       & \hspace{0.5in}-\frac{\eta_1}{2}\left(\vep^2 |\nabla_s \tu^1|^2 +|\partial_z \tu^1|^2\right) -\eta_2 W(\tu^1)\Bigr] 
       \textrm{ds}\,\textrm{dz}.
\end{align}
 A simple  calculation shows that
\begin{align}\label{kE1}
\lim\limits_{k\to\infty} \tF_{\vep_k}^\ell(\tu_k) =  \int_\Gamma \left( a_*^1 H_0^2(s) 
- (\eta_1+\eta_2) b_*^1\right)\textrm{ds}=G_1(\Gamma,U_*^1),
\end{align}
where $a_*^1=b_*^1>0$ are the corresponding constants
\begin{align} 
a_*^1 &= \frac12\int_{-L}^L  \left| \frac{\partial U_*^1}{\partial z}\right|^2\,dz, 
\qquad b_*^1=\int_{-L}^L W(U_*^1)dz.
\end{align}

\subsubsection{codimension-one sequence without codimension one energy} $\mbox{ }$
Pearling is a bifurcation characterized by 
rapid tangential oscillations in bilayer thickness. The existence of pearled solutions as critical points of the
FCH free energy has been established in a weakly nonlinear setting for a smooth well $W$, \cite{PW-15}. 
However this construction requires analysis of the linearization about the underlying bilayer solution and does not immediately 
extend to the case of a non-smooth well considered here.  However in the strongly nonlinear setting pearled solutions
may reduce to disconnected micelles, that is, codimension $n$ balls. These are radial solutions whose profile $U_*^n$ solves
\begin{align}
 \label{codim-n}
 \partial_R^2U_*^n + \frac{n-1}{R}\partial_R U_*^n = W'(U_*^n),
 \end{align}
 where $R$ is the $\varepsilon$ scaled distance to a center point. The codimension $n$ profile $U_*^n$ 
 satisfies $\partial_RU_*^n(0)=0$ and has support
contained within $R\in[0,R_0]$. For a fixed codimension one interface $\Gamma$, we form a codimension-one
$\ell$ bounded sequence with $\ell>R_0$. Taking $\{\vep_j\}$ tending to zero as $j\to\infty$ and for each $j$ identify 
$N_j$ points $\{x_{j,k}\}_{k=1}^{N_j}$  on $\Gamma$ whose separation is greater than $\vep_j R_0$. Since $\Gamma$ is far from self intersection there exists $\alpha_0>0$ sufficiently small such that for each $\alpha\in(0,\alpha_0)$ we may choose the points so that
 $N_j \sim\alpha \eps_j^{1-n}$. We form the sequence $\{w_j\}$ according to the formula
\begin{align}
 w_j(x) = \sum_{k=1}^{N_{j}} U_*^n\left(\frac{|x-x_{j,k}|}{\vep_j}\right).
 \end{align}

 To evaluate the energy we first consider, $U_*^n$. Multilplying (\ref{codim-n})
 by $\partial_R U_*^n$, integrating from $s=R$ to $s=\infty$, and using $W(U_*^n(\infty))=W(0)=0$, we obtain
\begin{align}
 \frac12 |\partial_R U_*^n|^2 - (n-1)\int\limits_R^\infty \frac{1}{s} |\partial_R U_*^n(s)|^2\,ds = W(U_*^n).
 \end{align}
Multiplying this expression by $R^{n-1}$ and integrating over  the region $R=0$ to $R=\infty$ yields
\begin{align}
\int\limits_0^\infty W(U_*^n)R^{n-1}\, dR = \int\limits_0^\infty\left[ \frac{R^{n-1}}{2} |\partial_R U_*^n|^2 - 
\frac{n-1}{n}\partial_R\left(R^n\right)\int\limits_R^\infty \frac{1}{s}|\partial_R U_*^n(s)|^2\, ds
\right] dR.
\end{align}
The integral over $s$ is identically zero for $R$ beyond the support of $U_*^n$, while it is bounded as 
$R\rightarrow0.$ We may integrate
by parts on the second term on the right hand side, obtaining
\begin{align}
  \int\limits_0^\infty W(U_*^n)R^{n-1}\, dR = \frac{2-n}{2n} \int\limits_0^\infty |\partial_R U_*^n|^2 R^{n-1}\,dR 
  = \frac{2-n}{2n}\sigma_n,
  \end{align}
  where we have introduced the codimension-$n$ surface tension 
  \beq\label{def-sigma-n}
  \sigma_n:=  \int_0^\infty |\partial_R U_*^n|^2R^{n-1}dR.
  \eeq
With these results it is easy to see that
 \begin{align}
 F_\epsilon(U_*^n) =  -\epsilon^{n-1} \left(\frac{\eta_1}{2}+\frac{2-n}{2n}\eta_2\right)\sigma_n,
 \end{align}
and hence 
 \begin{align}
 \lim\limits_{j\to\infty} \tF_\vep^\ell(w_j) =  -\alpha \left(\frac{\eta_1}{2}+\frac{2-n}{2n}\eta_2\right)\sigma_n.
 \end{align}
 Moreover there exists a constant $c>0$ such that 
 \begin{align}
  \left\| w_j\right\|_{L^2(\Omega_b)}&\geq c,\nn\\
   \left\| \nabla_s w_j\right\|_{L^2(\Omega_b)}& \geq c \vep_j^{-1}, \nn\\
   \left\| \Delta_s w_j\right\|_{L^2(\Omega_b)} & \geq c \vep_j^{-2},
 \end{align}
 for all $j\geq 1.$ The codimension-one $\ell$-bounded sequence satisfies \qref{T1-B1} but does not satisfy the enhanced bounds
\qref{T1-B2} nor \qref{T2-B3}. In particular it is straightforward to choose the points $\{x_{j,k}\}$ so that no subsequence 
of $\{w_j\}_{j=1}^\infty$ converges strongly in $L^2(\Omega_b)$.

\section{Discussion}

We have shown that sequences of FCH-energy bounded functions whose support converges to the same codimension one interface, $\Gamma$, may have fundamentally different
structure and that their limiting energy can target different features of the underlying interface. Consequently,  the relative size of the corresponding limiting energy 
can be exchanged under subtle changes in the parameters in the FCH energy. Specifically,
the bilayer sequence $\{\tu_k\}$ constructed in section 4.2.1 and the micelle sequence $\{\tw_k\}$ constructed in section 4.2.2 
are both codimension-one $\ell$-bounded for any codimension-one interface $\Gamma$ satisfying the curvature bounds \qref{Gamma-bounds}. 
If $\eta_2=-\eta_1<0$, then the energy of the bilayer sequence converges to a positive number. 
Indeed from Corollary 4.2, any codimension-one $\ell$-bounded sequence that satisfies
the enhanced tangential derivative bounds \qref{T1-B2} nor \qref{T2-B3} has its energy bounded below by a positive lower bound, 
since $G_1(\Gamma,\tu^*)>0$ for all $\tu\in\cS(\Gamma,\ell)$ when $\eta_2<-\eta_1.$ The energy of the 
micelle based sequence $\{\tw_k\}$ has a negative limit, $-\alpha(1-\frac{1}{n})\eta_1\sigma_n$ in space dimension $n\geq 2.$
Conversely, if $\eta_2>\frac{n}{n-2}\eta_1>0$, then the micelle sequence has a positive energy limit while the bilayer sequence has a negative
energy limit if the curvatures of $\Gamma_0$ are sufficiently small.   

The attempt to obtain well defined lower bounds to free energy by restricting the support of $u$ to lie in a thin codimension-one
domain is frustrated by the fact that the FCH free energy supports higher codimensional structures that can be naturally embedded within
a codimension one domain. However, we have shown that for sequences whose tangential variation is sufficiently tame and whose
$H^1$ weak limits lies in $H^2(\Omega_1^\ell)$, then the possible limit set corresponds to $N$ translates of a bilayer, with
a corresponding limiting energy. It is natural to extend this analysis to restrict the support of $u$ to lie in a thin neighborhood of a
codimension-$m$ domain, including codimension two filamentous pores in $\R^3.$ 
It is also important to characterize defect structures such as triple junctions, and open edges. 

\bibliographystyle{plain}
\bibliography{dai-bib}

\bigskip

{\it E-mail address:} sdai4@ua.edu  

{\it E-mail address:} PROMISLO@msu.edu

\end{document}